\documentclass[letterpaper, 10pt, conference]{ieeeconf}

\usepackage{graphics} 
\usepackage{amsmath} 
\usepackage{amssymb}  
\usepackage{cite}
\usepackage{bm}
\usepackage{acronym}
\usepackage{paralist}
\usepackage{float}
\usepackage{color}
\usepackage{epstopdf}
\usepackage{multicol}
\usepackage{tikz}

\newtheorem{FTS}{Theorem}
\newtheorem{Bhat}[FTS]{Theorem}

\newtheorem{NSC}[FTS]{Theorem}
\newtheorem{NC}[FTS]{Theorem}

\newtheorem{PathUni}[FTS]{Theorem}
\newtheorem{FTSLS}[FTS]{Theorem}
\newtheorem{FT Cons Col}[FTS]{Theorem}
\newtheorem{LT FTS}[FTS]{Theorem}

\newtheorem{FT Conv}[BF]{Lemma}
\newtheorem{BF bound}[BF]{Lemma}
\newtheorem{BF grad NZ}[BF]{Lemma}
\newtheorem{BF grad}[BF]{Lemma}
\newtheorem{Sclr FT}[BF]{Lemma}

\DeclareMathOperator{\sign}{sign}
\DeclareMathOperator{\F}{\mathrm F}

\IEEEoverridecommandlockouts
\begin{document}

\title{\LARGE \bf New Results on Finite-Time Stability: Geometric Conditions and Finite-Time Controllers}
\author{Kunal Garg \and Dimitra Panagou 
\thanks{The authors are with the Department of Aerospace Engineering, University of Michigan, Ann Arbor, MI, USA; \texttt{\{kgarg, dpanagou\}@umich.edu}.}
\thanks{The authors would like to acknowledge the support of the NASA Grant NNX16AH81A.}}

\maketitle

\begin{abstract}
This paper presents novel controllers that yield finite-time stability for linear systems. We first present a sufficient condition for the origin of a scalar system to be finite-time stable. Then we present novel finite-time controllers based on vector fields and barrier functions to demonstrate the utility of this geometric condition. We also consider the general class of linear controllable systems, and present a continuous feedback control law to stabilize the system in finite time. Finally, we present simulation results for each of these cases, showing the efficacy of the designed control laws. 
\end{abstract}

\acrodef{wrt}[w.r.t.]{with respect to}
\acrodef{apf}[APF]{Artificial Potential Fields}
\acrodef{ges}[GES]{Globally Exponentially Stable}
\newcounter{mytempeqncnt}

\section{Introduction}
Finite Time Stability (FTS) has been a well-studied concept, motivated in part from a practical viewpoint due to properties such as achieving convergence in finite time, as well as exhibiting robustness \ac{wrt} disturbances \cite{ryan1991finite}. Classical optimal control theory provides several examples of systems that exhibit convergence to the equilibrium in finite time \cite{ryan1982optimal}. A well-known example is the double integrator with bang-bang time-optimal feedback control \cite{athans2013optimal}; these approaches typically involve solutions that render discontinuous system dynamics. The approach in \cite{coron1995stabilization} considers finite-time stabilization using time-varying feedback controllers. The authors in \cite{bhat2000finite} focus on continuous autonomous systems and present Lyapunov-like necessary and sufficient conditions for a system to exhibit  FTS, while in \cite{bhat2005geometric} they provide geometric conditions for homogeneous systems to exhibit FTS. \cite{haddad2008finite} extended the notion of finite-time stability from autonomous to time-varying dynamical systems, see also \cite{hui2009semistability}. The authors in \cite{haimo1986finite} provided necessary and sufficient geometric conditions for the finite-time stability of a scalar system, and used the structure of phase portraits for second order systems to develop a class of finite-time systems. In \cite{wang2010finite}, the authors presented a method to construct a finite-time consensus protocol. \cite{amato2010finite} presents necessary and sufficient conditions for FTS of linear, time-varying systems, as well as an output feedback controller that yields finite-time stability. \cite{hong2002finite} addresses the problem of FTS for small-time controllable systems. FTS has regained much attention in the recent few years as well; \cite{liu2016nonsmooth,li2015finite,wang2014distributed} present FTS results for neural-network systems, output feedback tracking and control of multi-agent systems, respectively. In \cite{wang2016fault}, the authors consider the problem of finite-time consensus and provide a method to bound the position and velocity errors to a small residual set in finite-time. In \cite{zheng2012finite}, the authors analyze the finite-time consensus problem for strongly connected graphs of heterogeneous systems. Other recent work includes \cite{li2016results,li2013robust}, in which finite-time stability is studied in hybrid systems framework.

In \cite{haimo1986finite}, the authors presented a sufficient geometric condition in terms of the integral of the multiplicative inverse of the system dynamics, evaluated between any initial point $p$ and the origin. In this paper, we present a necessary condition in terms of the derivative of the system dynamics evaluated at the origin, which is much easier to check than the former one. We also present a sufficient condition in terms of bounds on the system dynamics and utilize it to design finite-time controllers for different classes of systems. In addition, we consider a general class of linear controllable systems, whereas the aforementioned work considered a very special class of linear or nonlinear systems. 

In \cite{bhat1998continuous}, the authors considered the problem of finite-time stabilization of double integrator systems. In this paper, we prove that under the effect of our controller, the closed-loop trajectories of any linear controllable system would converge to the equilibrium point in finite-time. As case studies, we consider a nonholonomic system guided by a vector field, a single integrator system guided by a barrier-function based controller, and controllable LTI system stabilized at an arbitrary equilibrium point, and design finite-time controllers for each one of them. 


The paper is organized as follows: Section \ref{Sec OV} presents an overview of the theory of FTS. In Section \ref{New GC} we present new geometric conditions to establish FTS for scalar systems. Section \ref{BF control} presents a finite-time Barrier function based controller for obstacle avoidance and convergence to the goal location. In Section \ref{Sec LS}, we present novel control laws for a class of linear controllable systems for FTS. Section \ref{Simulations} evaluates the performance of the proposed finite-time controllers via simulation results. Our conclusions and thoughts on future work are summarized in Section \ref{Conclusions}.
 
\section{Overview of Finite Time Stability}\label{Sec OV}
Let us consider the system: 
\begin{align}\label{ex sys}
\dot y = f(y(t)),
\end{align}
where $y\in \mathbb R$, $f: \mathbb R \rightarrow \mathbb R$ and $f(0)=0$. In \cite{bhat2000finite}, the authors define finite-time stability as follows: 
The origin is said to be a finite-time-stable equilibrium of \eqref{ex sys} if there exists an open neighborhood $\mathcal N \subset D$ of the origin and a function $T: \mathcal N \setminus\{0\} \rightarrow (0,\infty)$, called the settling-time function, such that the following statements hold:
\begin{enumerate}
    \item \textit{Finite-time convergence}: For every $x \in \mathcal N \setminus\{0\}$, $\psi^x$ is defined on $[0, T(x)), \psi^x(t) \in \mathcal N \setminus\{0\}$ for all $t \in [0, T(x))$, and $\lim_{t\to T(x)} \psi^x(t)=0$. Here, $\psi^x: [0, T(x))\rightarrow D$ is the unique right maximal solution of system \eqref{ex sys}.  
    \item \textit{Lyapunov stability}: For every open neighborhood $U_\epsilon$ of 0, there exists an open subset $U_\delta$ of $\mathcal N$ containing 0 such that, for every $x \in  U_\delta\setminus\{0\}, \; \psi^x(t) \in U_\epsilon$, for all $t \in [0, T(x))$.
\end{enumerate}
The origin is said to be a globally finite-time-stable equilibrium if it is a finite-time-stable equilibrium with $D = \mathcal N = \mathbb R^n$. The authors also presented Lyapunov like conditions for finite-time stability of system \eqref{ex sys}:
\begin{Bhat} \cite{bhat2000finite}\label{FTS Bhat}
Suppose there exists a continuous function $V$: $\mathcal{D} \rightarrow \mathbb{R}$ such that the following hold: \\
 (i) $V$ is positive definite \\
 (ii) There exist real numbers $c>0$ and $\alpha \in (0, 1)$ and an open neighborhood $\mathcal{V}\subseteq \mathcal{D}$ of the origin such that 
 \begin{align} \label{FTS Lyap}
     \dot V(y) + c(V(y))^\alpha \leq 0, \; y\in \mathcal{V}\setminus\{0\}.
 \end{align}
Then origin is finite-time stable equilibrium of \eqref{ex sys}. 
\end{Bhat}

\subsection{Notations}
We denote $\|\bm x\|$ the Euclidean norm $\|\bm x\|_2$ of vector $\bm x$, and $|x|$ the absolute value of the scalar $x$. The $\sign(x)$ function is defined as:
\begin{align}
    \sign(x) = \left\{
	\begin{array}{rc}
	- 1, & \hbox{$x < 0$;}\\
	0, & \hbox{$x = 0$;} \\
	1, & \hbox{$x > 0$.} \\
	\end{array}
	\right.
\end{align}
\section{New condition for Finite-time stability}\label{New GC}
\subsection{Geometric Conditions for FTS}
The authors in \cite{haimo1986finite} stated a geometric condition on the system dynamics for the equilibrium to be finite-time stable:
\begin{enumerate}
\item $yf(y) < 0$ for $y \in \mathcal N\setminus\{0\}$, and $yf(y) = 0$ when $y = 0$, and
\item $\int_{p}^{0} \frac{dy}{f(y)} < \infty$ for all $p\in \mathbb R$. 
\end{enumerate}
These conditions are not useful in practice as, in general, it is difficult to evaluate the integral $\int_{p}^{0} \frac{dy}{f(y)}$ for an arbitrary vector field $f(y)$. Hence, we present conditions which are easier to check. Note that these results follow immediately from \cite{haimo1986finite}. Before presenting our condition, we state the following well-known result

\begin{NC}\label{suf con}
If system \eqref{sys0} is finite time stable, then 
\begin{align}\label{h inf}
\left.\frac{\partial h(x)}{\partial x}\right \vert _{x = 0}  = -\infty    
\end{align}
\end{NC}
This has been pointed out by several authors, see for e.g. \cite{haimo1986finite}. Note that this is not a sufficient condition: take $x(t) = x_0e^{-t^2}$ as a counter example. This system goes to origin only as $t\rightarrow \infty$.
Taking its first and second time derivative, we get $\dot x (t) = -2tx_0e^{-t^2}$ and $\ddot x(t) = (-2+4t^2)x_0e^{-t^2}$. Now, condition \eqref{h inf} can be re-written in the following form:
\begin{align*}
    \left.\frac{d h(x)}{dx}\right\vert _{x = 0} =  \left.\frac{d^2 x(t)}{d t^2}\frac{d t}{d x(t)}\right\vert _{x = 0}
\end{align*}
For the above example, $x = 0$ at $t = \infty$. Using this and substituting the expressions of $\dot x(t)$ and $\ddot x(t)$, we get
\begin{align*}
    \left.\frac{d h(x)}{dx}\right\vert _{x = 0} =  \left.\frac{(-2+4t^2)x_0e^{-t^2}}{-2tx_0e^{-t^2}}\right\vert _{t = \infty} = \lim_{t\to\infty}-2t = -\infty.
\end{align*}
Hence, even though the system $x(t) = x_0e^{-t^2}$ is not finite-time stable, it satisfies \eqref{h inf}. This means that the vector field $h(x)$ cannot be Lipschitz continuous at the origin for system to be FTS. Now we present a sufficient condition for the origin of a scalar system to be finite-time stable:
\begin{NSC}\label{nec suf cond}
Consider the system:
\begin{align}
        \dot x & = h(x), \quad x \in D \subset \mathbb R \label{sys0},
\end{align}
such that $h(0) = 0$, and $x\; h(x) < 0 \quad \forall x \neq 0$, i.e., the origin is a stable equilibrium. Then the origin is finite time stable equilibrium for system \eqref{sys0} if : $\exists \; D\subset \mathbb R$ containing the origin such that $\forall x\in D$, 
\begin{align} \label{nec suf ineq}
\sign(x)h(x)  \leq -k|x|^\alpha, \; \mbox{where} \; k>0, \; 0<\alpha<1.    
\end{align}
\end{NSC}
\begin{proof} Choose the candidate Lyapunov function $V(x) = \frac{1}{2}x^2$. Taking its time derivative along the trajectories of \eqref{sys0}, we obtain: 
\begin{align*}
 \dot V = x\; h(x) = |x|\sign(x)h(x).   
\end{align*}
Since $\sign(x) h(x) \leq -k|x|^\alpha, \; k>0, \; 0<\alpha<1$, we have: $\dot V \leq |x|(-k|x|^\alpha)$. Choosing $\beta = \frac{1+\alpha}{2}$ and $c = k2^\beta$, we get
\begin{align*}
 \dot V \leq -cV(x)^\beta    
\end{align*}
where $0<\beta <1$ and $c>0$. Hence, from Theorem \ref{FTS Bhat}, we get that the origin is finite-time stable. 

\end{proof}
Note that this condition is not necessary but just sufficient. Example 2.3 in \cite{bhat2000finite} is a system with Finite-time stable origin but there exists no $k>0$ and $0<\alpha<1$ in any open neighborhood or origin $D\subset \mathbb R$ such that the inequality \eqref{nec suf ineq} holds. 
These conditions can be used to verify finite-time stability of scalar systems. Now we present some examples to demonstrate how these conditions can be utilized to design finite-time controllers.

\subsection{Example: Trajectory Tracking}\label{Sec 2 Path}
Consider a vehicle modeled under unicycle kinematics given by:
\begin{align} \label{uni-sys}
   \begin{bmatrix}\dot x\\ \dot y \\ \dot \theta \end{bmatrix}=\begin{bmatrix}u \cos\theta \\u  \sin\theta \\ \omega \end{bmatrix},
\end{align}
where $\bm{q}=\begin{bmatrix}\bm r^T&\theta\end{bmatrix}^T \in X \subset \mathbb R^3$ is the state vector of the vehicle, comprising the position vector $\bm{r}=\begin{bmatrix}x&y\end{bmatrix}^T$ and the orientation $\theta$ of the agent \ac{wrt} the global frame $\mathcal G$, $\bm u=\begin{bmatrix}u&\omega\end{bmatrix}^T \in \mathcal U \subset \mathbb R^2$ is the control input vector comprising the linear velocity $u$ and the angular velocity $\omega$ of the vehicle. The control objective is to track a trajectory given by $\bm r_g(t)$ which is continuously differentiable in its argument. We seek vector field based controller which can converge the trajectories of closed-loop system to the desired trajectory in finite-time. First we design a vector field to achieve this:
\begin{align}\label{F-traj}
    \mathbf F_p = - k\bm r_e(t)\|\bm r_e(t)\|^{\alpha-1} +  \dot{\bm r}_g(t) ,
\end{align}
where $k>0$, $0<\alpha<1$ and $\bm r_e(t) = \bm r(t) - \bm r_g(t)$. The control law is given by:
\begin{align}
    u &= \|\mathbf F_p\|, \label{path-u}\\
    \omega &= -k_\omega\sign\left(\theta-\varphi_p\right)|\theta-\varphi_p|^\alpha+\dot \varphi_p, \label{path-w}
\end{align}
where $\varphi_p\triangleq\arctan\left(\frac{\F_{py}}{\F_{px}}\right)$ is the orientation of the vector field $\mathbf F_p$. 

\begin{PathUni}\label{PathUni}
Under control law \eqref{path-u}-\eqref{path-w}, system \eqref{uni-sys} tracks the trajectory $\bm r_g(t)$ in finite time.
\end{PathUni}
Before stating the main theorem, we present an intermediate result that is used also later in the paper:
\begin{FT Conv}\label{FT lemma}
Origin of the following system is a finite-time stable equilibrium:
\begin{align}\label{gen FTS}
    \dot{\bm x} =  -k\bm x\|\bm x\|^{\alpha-1} \quad k>0 \quad 0<\alpha<1. 
\end{align}
\end{FT Conv}
\begin{proof}
Consider the candidate Lyapunov function 
\begin{align*}
    V(\bm x) = \frac{1}{2}\|\bm x\|^2.
\end{align*}
The time derivative of this function along the system trajectories of \eqref{gen FTS} read:
\begin{align*}
    \dot V(\bm x) &= \bm x^T( -k\bm x\|\bm x\|^{\alpha-1}) = -k\|\bm x\|^{1+\alpha} = -cV(\bm x)^\beta,
\end{align*}
where $c = 2^{\frac{1+\alpha}{2}}k >0$ and $\beta = \frac{1+\alpha}{2} <1$. Hence, from Theorem \ref{FTS Bhat}, we get that the equilibrium point $\bm 0$ is finite-time stable. Note that for scalar case, right hand side simply reads $-kx|x|^{\alpha-1} = -k\sign(x)|x|^{\alpha}$.
\end{proof}
Now we prove Theorem \ref{PathUni}:

\begin{proof}
Consider error term $\bm r_e(t) = \bm r(t) - \bm r_g(t)$. Its time derivative reads $\dot{\bm r}_e(t) = \dot{\bm r}(t) - \dot{\bm r}_g(t) = \begin{bmatrix}u\cos\theta \\ u\sin\theta\end{bmatrix} - \dot{\bm r}_g(t)$. From Lemma \ref{FT lemma}, $\bm r_e(t)$ goes to origin in finite-time if $\dot{\bm r}_e(t) = -k\bm r_e\|\bm r_e\|^{\alpha-1}$ with $k>0$ and $0<\alpha<1$. Hence, we need that 
\begin{align}\label{des-cl}
\begin{bmatrix}u_d\cos\theta_d \\ u_d\sin\theta_d\end{bmatrix}  = -  k\bm r_e\|\bm r_e\|^{\alpha -1} + \dot{\bm r}_g(t)
\end{align}
where $u_d$ and $\theta_d$ denote the desired linear speed and orientation, respectively. Let $\angle(\cdot)$ denote signed angle. From \eqref{F-traj} and \eqref{des-cl}, we have that $\varphi_p = \angle\mathbf F_p = \theta_d$ and $u = u_d = \|\mathbf F_p\|$. Hence, if the system tracks vector field $\mathbf F_p$ in finite-time, it will track the desired trajectory $\bm r_g(t)$ in finite-time. Define $\theta_e = \theta - \theta_d$. Choose candidate Lyapunov function  $V(\theta_e) = \frac{1}{2}\theta_e^2$. Taking its time derivative along \eqref{path-w}, we get:
\begin{align*}
\dot V(\theta_e) & = \theta_e\dot{\theta}_e  = \theta_e(\dot \theta - \dot \theta_d) =  \theta_e(\omega-\dot \theta_d)  \\ 
& \overset{\eqref{path-w}}{=} \theta_e(-k_\omega\sign\left(\theta-\varphi_p\right)|\theta-\varphi_p|^\alpha+\dot \varphi_p - \dot\theta_d).
\end{align*}
Since $\theta_d = \varphi_p$, we get:
\begin{align*}
\dot V(\theta_e) & = \theta_e(-k_\omega\sign\left(\theta-\varphi_p\right)|\theta-\varphi_p|^\alpha) \\
& = -k_\omega\theta_e\sign(\theta_e)|\theta_e|^\alpha = -k_\omega|\theta_e|^{1+\alpha} \\
& = -k_\omega(2V(\theta_e))^{\frac{1+\alpha}{2}} \leq -cV(\theta_e)^\beta,
\end{align*}
where $c = k_\omega2^{\frac{1+\alpha}{2}}$ and $\beta = \frac{1+\alpha}{2}<1$. Hence, $\theta_e(t)$ converges to zero in finite time. This along with the fact that the magnitude of linear speed given out of \eqref{path-u} is equal to the desired linear speed $u_d$ implies that $\dot{\bm r}(t) = -k\bm r_e(t)\|\bm r_e(t)\|^{\alpha-1}$. Hence, $\bm r_e(t) \rightarrow \textbf{0}$ in finite time and which implies that the system trajectory $\bm r(t)$ converges to $\bm r_g(t)$ in finite-time. 
\end{proof}
Note that for the error dynamics of orientation $\theta_e$, from \eqref{path-w}, one can get 
\begin{align*}
&\dot \theta_e = -k_\omega\sign(\theta_e)|\theta_e|^\alpha = h(\theta_e) \\
\implies & \theta_eh(\theta_e) = -k_\omega|\theta_e|^{1+\alpha} \\
\implies & \sign(\theta_e)h(\theta_e) = -k_\omega|\theta_e|^\alpha
\end{align*}
Also, $ \frac{dh(\theta_e}{d \theta_e} = -k_\omega \sign(\theta_e)|\theta_e|^{\alpha-1} $ which implies $ \left. \frac{d h(\theta_e)}{d\theta_e}\right \vert _{\theta_e = 0} = -\infty $ since $\alpha<1$. Hence, both the conditions presented in Theorem \ref{nec suf cond} and \ref{suf con} are getting satisfied.

\section{Obstacle Avoidance Using Barrier Function Based Controller}\label{BF control}
Consider a vehicle modeled under single integrator dynamics as
\begin{align}\label{sin int}
\dot{\bm x} = \bm u,
\end{align}
where $\bm x, \bm u \in \mathbb R^n$. The problem of reaching to a specified goal location in finite time can be formulated mathematically as follows:
$$\exists t^*<\infty \ s.t \ \forall t\geq t^* \ \|\bm x(t) - \bm \tau\| = 0 $$ where  $\bm \tau$ is the desired goal location, while that of obstacle avoidance can be written as:
$$ \|\bm x(t) - \bm o\| \geq d_c \ \forall t\geq t_0,$$
where $\bm o$ represents the location of the obstacle and $t_0$ is the starting time. Here, we model the obstacle as circular discs of radius $\rho_o$. The vehicle is required to maintain a safe distance $d_m$ from the obstacle. Hence choosing $d_c = d_m + \rho_o$ implies that vehicle maintains the required minimum distance from the boundary of the obstacle if $\|\bm x - \bm o\| \geq d_c$. We assume that the obstacle is located in such a manner that $ \|\bm o- \bm \tau\| > 2d_c$ so that at the desired location is sufficiently far away from the obstacle. We also assume that agent starts sufficiently far away from the obstacle so that $\|\bm x(t_0)-\bm o\|>d_c$. We seek a continuous feedback-law $\bm u_i$ such that the system \eqref{sin int} reaches its goal location while maintaining safe distance from the obstacle. More specifically, we seek a Barrier-function based controller for this problem. First we define the Barrier function as follows:
\begin{align}\label{BF eq}
B (\bm x)= \frac{\|\bm x - \bm \tau\|^2}{\|\bm x-\bm o\|-d_c +\frac{1}{\epsilon}}, 
\end{align}
where $\epsilon \gg 1$ is a very large number. We choose the controller of the form:
\begin{equation}\label{u consens}
\bm u = -k_1\nabla B\|\nabla B(\bm x)\|^{\alpha-1},
\end{equation}
where $k_1 >0$ and $0<\alpha<1$. With this controller, we have the following result:
\begin{FT Cons Col}\label{BF FTS}
Under the control law \eqref{u consens}, the point $\bm x=\bm \tau$ is FTS equilibrium for the closed-loop system \eqref{sin int}, and the closed-loop system trajectories will remain safe \ac{wrt} the obstacle. 
\end{FT Cons Col}

Before presenting the proof, we present some useful Lemmas:
\begin{BF bound}\label{BF B}
In the domain $D_o = \{\bm x \ | \ \|\bm x -\bm o\| > d_c \}$, the Barrier function $B(\bm x)$ is bounded as $B(\bm x) \leq \epsilon\|\bm x-\bm\tau\|^2$.
\end{BF bound}
\begin{proof}
In the chosen domain:
\begin{align*}
   & \|\bm x -\bm o\| \geq d_c \implies \|\bm x -\bm o\| - d_c\geq 0 \\
    \implies &\|\bm x -\bm o\| - d_c + \frac{1}{\epsilon} \geq \frac{1}{\epsilon}  \\
    \implies &\frac{1}{\|\bm x -\bm o\| - d_c + \frac{1}{\epsilon}} \leq \epsilon \\
    \implies & B(\bm x) = \frac{\|\bm x-\bm \tau\|^2}{\|\bm x -\bm o\| - d_c + \frac{1}{\epsilon}} \leq \epsilon\|\bm x-\bm \tau\|^2.
\end{align*}
\end{proof}
\begin{BF grad NZ}\label{BF NZ}
Gradient of the Barrier function, $\nabla B(\bm x)$ is non-zero everywhere except the equilibrium point $\bm \tau$ and at \begin{align}
    \bm x = \bm \tau + 2\frac{\|\bm o-\bm \tau\| + d_c-\frac{1}{\epsilon}}{\|\bm o-\bm \tau\|}(\bm o-\bm \tau)
\end{align}
\end{BF grad NZ}
\begin{proof}
Define $x_o \triangleq (\|\bm x-\bm o\|-d_c +\frac{1}{\epsilon})$. Gradient of the Barrier function \eqref{BF eq} reads:
\begin{align*}
    \nabla B(\bm x) = & 2\frac{\bm x -\bm \tau}{x_o} - \frac{\|\bm x-\bm \tau\|^2}{x_o^2}\frac{\bm x - \bm o}{\|\bm x-\bm o\|} 
\end{align*}
Hence, $\nabla B(\bm x) = \bm 0$ implies:
\begin{align*}
    \frac{\bm x-\bm \tau}{\|\bm x-\bm \tau\|} = \frac{\|\bm x-\bm \tau\|}{2x_o}\frac{\bm x - \bm o}{\|\bm x-\bm o\|},
\end{align*}
which holds only if $\bm x-\bm \tau$ is along $\bm x-\bm o$ and $\|\bm x-\bm \tau\| = 2x_o$ since left hand side of the equation is a unit vector. Denote $\|\bm x-\bm \tau\| = a$, $\|\bm x-\bm o\| = b$ and $\|\bm o-\bm \tau\| = c$. Since $\bm x-\bm o$ and $\bm x-\bm \tau$ are co-linear and due to the assumption $\|\bm o-\bm \tau\| >d_c$, we get that $a = b+c$. From that, we get
\begin{align*}
    & a = b+c = 2x_0 = 2b -2d_c + \frac{2}{\epsilon} \\
    \implies & b = c + 2d_c -\frac{2}{\epsilon}\Rightarrow a = 2(c+d_c-\frac{1}{\epsilon})
\end{align*}
From this, we get that 
\begin{align*}
    \bm x - \bm \tau & = 2(c+d_c-\frac{1}{\epsilon})\frac{\bm o-\bm \tau}{\|\bm o-\bm \tau\|}.
\end{align*}
As $c = \|\bm o-\bm \tau\| = \|\bm o_{\bm \tau}\|$ we get 
\begin{align*}
    \bm x = \bm \tau + 2\frac{\|\bm o_{\bm \tau}\| + d_c-\frac{1}{\epsilon}}{\|\bm o_{\bm \tau}\|}\bm o_{\bm \tau}.
\end{align*}
\end{proof}


\begin{BF grad}\label{BF grad}
In any closed, compact domain $D \subset \mathbb R^n$ containing point $\bm \tau$ and excluding the region $\bar D = \{\bm x \ | \ \|\bm x - \bm p\| \ < r \quad \bm p = \bm \tau + \theta(\bm o-\bm \tau) \ , \ \theta \geq 1\}$, where $r$ is an arbitrary small positive number, the gradient of Barrier function $B(\bm x)$ can be bounded as 
\begin{align}
\|\nabla B\| \geq c\|\bm x-\bm \tau\|,
\end{align}
where $c>0$.
\end{BF grad}

\begin{proof}
From \eqref{BF eq}, it can be easily verified that $\nabla B(\bm \tau) = \bm 0$.  Choose $D_1 = \{\bm x \ | \ \|\bm x-\bm \tau\| <\Delta\}$, where $\Delta$ is a very small positive number. Choose domain $\tilde D = D\setminus D_1$. Recall that $D$ doesn't include the ray $\bar D$, so $\tilde D$ does not include the point as in Lemma \ref{BF NZ}. Hence, from Lemma \ref{BF NZ}, at any point $\bm x \in \tilde D$, $\nabla B(\bm x) \neq 0$ and since $\tilde D$ is a closed domain, we can find $c_1 = \min_{\bm x\in \tilde D} \frac{\|\nabla B(\bm x)\|}{\|\bm x-\bm \tau\|} >0 $. Therefore, we have that $\forall \; \bm x\in \tilde D$, $\|\nabla B\| \geq c_1\|\bm x-\bm \tau\|$.

Now, consider $D_2 = \{\bm x \ | \ \|\bm x-\bm \tau\| \leq \Delta\}$. In a very small neighborhood of $\bm \tau$, the Hessian matrix $\nabla^2 B(\bm x) \succ \bm 0$ (i.e. $\nabla^2 B$ is a positive definite matrix). Therefore, using the gradient inequality (First-order condition for convexity), we have that $\forall\bm x \in D_2$, 
\begin{align*}
    & B(\bm \tau) \geq B(x) + \nabla B(\bm x)^T(\bm \tau - \bm x) \\
    \implies & 0 \geq B(\bm x) - \nabla B(\bm x)^T(\bm x -\bm \tau) \\
    \implies & \nabla B(\bm x)^T(\bm x -\bm \tau) \geq B(\bm x).
\end{align*}
From \eqref{BF eq}, one can easily see that $B(\bm x)$ can be bounded as $B(\bm x) \geq c_2\|\bm x-\bm \tau\|^2$. Also, using Cauchy-Schwartz inequality, we have that $\nabla B(\bm x)^T(\bm x -\bm \tau) \leq \|\nabla B(\bm x)\|\|\bm x-\bm \tau\|$. Therefore, we have that 
\begin{align*}
    \|\nabla B(\bm x)\|\|\bm x-\bm \tau\| & \geq \nabla B(\bm x)^T(\bm x -\bm \tau)\\
    & \geq  B(\bm x) \geq c_2\|\bm x-\bm \tau\|^2 \\
    \implies \|\nabla B(\bm x)\| & \geq c_2\|\bm x-\bm \tau\|
\end{align*}
Since $D = \tilde D \bigcup D_2$, choosing $c = min\{c_1,c_2\}$ gives us the required result.
\end{proof}
Now we are ready to prove Theorem \ref{BF FTS}:

\begin{proof}
From Lemma \ref{BF NZ}, we have that $\nabla B(\bm x) = \bm 0$ at the equilibrium point $\bm \tau$ and at the point $\bm x = \bm \tau + \mu (\bm o-\bm \tau)$ where $\mu$ takes the value as per Lemma \ref{BF NZ}. Lets assume that the initial condition is such that $\bm x(t_0)$ doesn't lie on the ray $\bar D$ defined as per Lemma \ref{BF grad}. Consider the open domain around the goal location $D_o$ as defined in Lemma \ref{BF B}. Define $\mathcal D = D_o\setminus \bar D$ (see figure \ref{fig:BF Scene}. Since $\bar D$ is a closed domain and $D_o$ is open, domain $\mathcal{D}$ is an open domain around the equilibrium $\bm\tau$. 

\begin{figure}[h]
	\centering
	\includegraphics[width=0.8\columnwidth,clip]{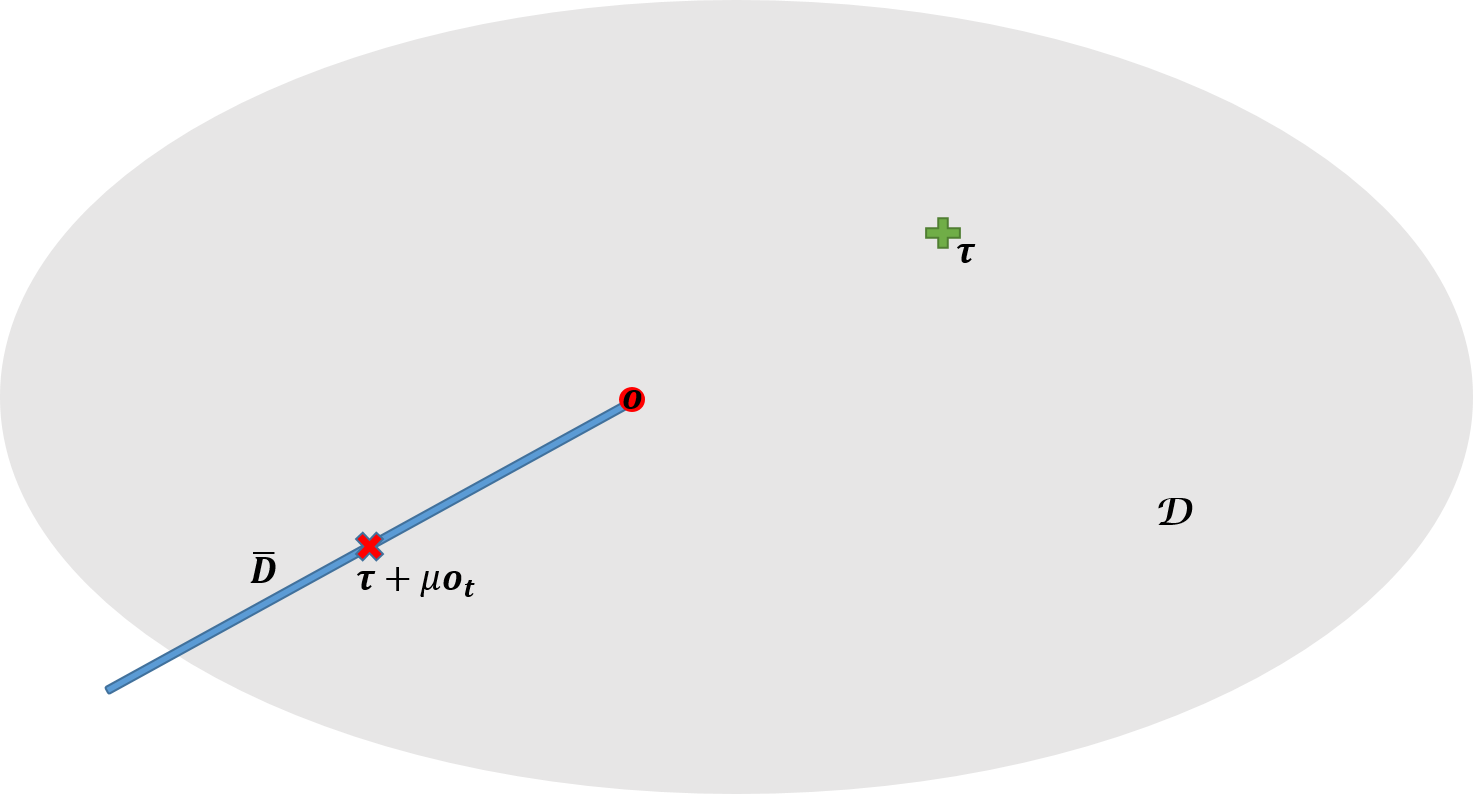}
	\caption{The working domain $\mathcal{D}$ and the excluded region $\bar D$.}
	\label{fig:BF Scene}
\end{figure}

Choose the candidate Lyapunov function $$V(\bm x) =  B(\bm x).$$ For simplifying the notation, define $\bm x_e = \bm x - \bm \tau$ and drop the argument $\bm x$ for the functions $B$ and $\nabla B$. Taking the time derivative of $V(\bm x)$ along the trajectories of \eqref{sin int}, we have:
\begin{align*}
  \dot V(\bm x) = &(\nabla B)^T(-k_1\nabla B\|\nabla B\|^{\alpha-1})\\
    = &-k_1\|\nabla B\|^{1+\alpha}
\end{align*}
From Lemma \ref{BF grad}, we have that $\|\nabla B\| \geq c\|\bm x_e\|$. Using this, we get:
\begin{align*}
\dot V \leq & -k_1c^{1+\alpha}\|\bm x\|^{1+\alpha}
\end{align*}
Using the result from Lemma \ref{BF B}, we have that $B \leq \epsilon\|\bm x_e\|^2$. Hence, we get:
\begin{align*}
    & V(\bm x) = B(\bm x) \leq \epsilon \|\bm x_e\|^2 \Rightarrow \\
     & -\|\bm x_e\|^2 \leq -\frac{1}{\epsilon}V(\bm x) \Rightarrow \\
     & -\|\bm x_e\|^{1+\alpha} \leq -\frac{1}{\epsilon^{\frac{1+\alpha}{2}}}V(\bm x)^{\frac{1+\alpha}{2}}
\end{align*}
Hence, we have that: $$\dot V(\bm x) \leq -k_1c^{1+\alpha}|\bm x_e\|^{1+\alpha}  \leq -KV(\bm x)^\beta, $$
for any $\bm x \in \mathcal D$, where $K = (\frac{1}{\epsilon})^\beta k_1c^{1+\alpha}$ and $\beta = \frac{1 + \alpha}{2}<1$. Thus, we have that the equilibrium $\bm x = \bm \tau$ is finite-time stable. Safety is trivial from the construction of Barrier function: since $\dot V(\bm x)$ is negative, $V(\bm x) = B(\bm x)$ would remain bounded and hence, the denominator of $B(\bm x)$ would have non-zero positive value. Choose $\epsilon$ greater than 1 over this minimum non-zero value would ensure $\|\bm x-\bm o\|-d_c>0$.
\end{proof}

\section{finite time stability of LTI systems}\label{Sec LS}
Consider the system 
\begin{align}\label{sys 2}
    \bm{\dot x} &= \bm{Ax} + \bm{B}\bm u, 
\end{align}
where $\bm x \in \mathbb R^n \quad \bm u \in U \subset \mathbb R^n$, $\bm A, \bm B \in \mathbb R^{n \times n}$. 
Objective is to stabilize the origin of \eqref{sys 2} in finite time. Mathematically, we seek a continuous feedback law so that conditions of Theorem \ref{FTS Bhat} are satisfied. 
\subsection{Multi-Input Case}

\begin{FTSLS} \label{FTS Lin Sys}
Assume $\bm B$ is full rank in \eqref{sys 2}. Then, the feedback control law 
\begin{align}\label{u-FTS}
\bm u(\bm x) = \bm K_1\bm x + \bm K_2\bm x_\alpha
\end{align}
where $\bm K_1$ is such that $\bm A+\bm B\bm K_1$  is Hurwitz, $\bm K_2 = -\bm B^{-1}$ and 
\begin{align}\label{x-a}
\bm x_\alpha = \bm x\|\bm x\|^{\alpha-1},
\end{align}
with $0<\alpha<1$ stabilizes the origin of closed-loop system \eqref{sys 2} in finite-time. 
\end{FTSLS}
\begin{proof}
Since $\bm B$ is full rank, ($\bm A,\bm B$) is controllable. Therefore, there exists a gain $\bm K_1 $ s.t. $Re(eig(\bm{A+BK}_1))<0$ i.e., there exists a positive definite matrix $\bm Q$ s.t. $(\bm{A+BK}_1)^T\bm P+ \bm{P(A+BK}_1) = -\bm Q$ where $\bm P$ is a positive definite matrix. Choose candidate Lyapunov function $V(x) = \bm x^T\textbf{P}\bm x$. Taking time derivative of $V(x)$ along the closed-loop trajectories of system \eqref{sys 2}, we get:
\begin{align*}
    \dot V(\bm x) & = \bm{x^TP}(\bm{Ax+B}\bm u) + \bm{(Ax+B}\bm u)^T\bm{Px}\\
    & = \bm{x^T\bm P(Ax+BK}_1\bm x + \bm {BK}_2\bm x_\alpha) \\
    & + (\bm{Ax+BK}_1\bm x + \bm{BK}_2\bm x_\alpha)^T\bm{Px}\\
    &= \bm x^T(\bm P\tilde{\bm A}+\tilde{\bm A}^T\bm P)\bm x + \bm x^T\bm P\bm B\bm K_2\bm x_\alpha \\
    & + \bm x_\alpha^T\bm K_2^T\bm B^T\bm P\bm x
\end{align*}
where $\tilde{\bm A} = \bm A+\bm B\bm K_1$. Choose $\bm K_2 = -\bm B^{-1}$. Hence we have that $(\bm{PBK}_2+\bm K_2^T\bm B^T\bm P) = -2\bm P$.
Therefore
\begin{align*}
    \dot V(\bm x) & = -\bm x^T\bm Q\bm x -2\bm x^T\bm P\bm x_\alpha \leq -2\bm x^T\bm P\bm x_\alpha
\end{align*}
Now, for any symmetric matrix $\bm P$, $\bm x^T\bm P\bm y$ can be bounded as $\lambda_{min}(\bm P)\bm x^T\bm y\leq \bm x^T\bm P\bm y \leq \lambda_{max}(\bm P)\bm x^T\bm y$ if $\bm x^T\bm y>0$. From \eqref{x-a}, we have that $\bm x^T\bm x_\alpha>0$ for all $\bm x\neq 0$. Hence, we have that $\bm x^T\bm P\bm x_\alpha \geq \lambda_{min}(\bm P)\bm x^T\bm x_\alpha$. Also, we can bound $V(x)$ by 
\begin{align*}
    &\lambda_{min}(\bm P)\|\bm x\|^2 \leq  \bm x^T\bm P\bm x \leq \lambda_{max}(\bm P)\|\bm x\|^2 \\
    &\implies V(\bm x)  \leq \lambda_{max}(\bm P)\|\bm x\|^2 \\
    & \implies V(\bm x)^\beta  \leq k \|\bm x\|^{2\beta}  =   k \|\bm x\|^{1+\alpha}
\end{align*}
where $k = (\lambda_{max}(\bm P))^\beta>0$ and $\beta = \frac{1+\alpha}{2}<1$. Hence
\begin{align*}
    \dot V(\bm x) &\leq -2\bm x^T\bm P\bm x_\alpha \leq -2\lambda_{min}(P)\bm x^T\bm x_\alpha \\
    & = -2\lambda_{min}(\bm P) \bm x^T\bm x\|\bm x\|^{\alpha-1} = -2\lambda_{min}(\bm P) \|\bm x\|^{1+\alpha} \\
    &\leq -2\lambda_{min}(\bm P)\frac{V(\bm x)^\beta}{k} \leq -cV(x)^\beta \\
    \implies \dot V(\bm x) &+ cV(x)^\beta \leq 0 
\end{align*}
where $c = \frac{2\lambda_{min}(\bm P)}{k}> 0$ and $0<\beta<1$ since $0<\alpha<1$. Therefore, from Theorem \ref{FTS Bhat}, we have that the origin is finite-time stable for closed-loop system.
\end{proof}

This was a restrictive case since we assumed matrix $\mathbf B$ to be full rank. Before we present the most general case, we state a result that we would require:
\begin{Sclr FT}\label{Lemma FT}
Consider the scalar system 
\begin{align}
    \dot x = ax + bu \quad b\neq 0.
\end{align}
$x$ converges to any $C^1$ trajectory given by $x^d(t) \in \mathbb R$ in finite time with control law
\begin{align}\label{scrl u}
    u = \frac{1}{b}(-ax -k\sign(x-x^d)|x-x^d|^\alpha + \dot x^d(t)),
\end{align}
where $k>0$ and $0< \alpha <1$.  
\end{Sclr FT}

\begin{proof}
Choose candidate Lyapunov function $V(x) = \frac{1}{2}(x-x^d)^2$. The time derivative of this function along the closed loop trajectory reads
\begin{align*}
    \dot V(x) & = (x-x^d)(\dot x-\dot x^d) \\
    & = (x-x^d)(ax + bu-\dot x^d) \\
    & \overset{\eqref{scrl u}}{=} (x-x^d)(-k\sign(x-x^d)|x-x^d|^\alpha)\\
    & = -k |x-x^d|^{1+\alpha} = -cV(x)^\beta,
\end{align*}
where $c = k2^\frac{1+\alpha}{2}$ and $\beta = \frac{1+\alpha}{2} <1$. Hence, from Theorem \ref{FTS Bhat}, we get that $x(t)-x^d(t)\rightarrow 0$ in finite time. 
\end{proof}
Now we present the most general case of linear controllable systems and show that any controllable system can be controlled in finite time.

\subsection{Linear Controllable System}\label{LCS}
Since the system is assumed to be controllable, we start with controllable canonical form:
\begin{align}\label{lin-sys}
    \dot{\bm x} = \bm A\bm x+\bm Bu,
\end{align}
where $\bm x\in \mathbb R^n$ and $u\in \mathbb R$ and system matrices have form $$\bm A = \begin{bmatrix}0 & 1 & 0 &\dots & 0 & 0\\ 0& 0& 1& 0& \cdots &0 \\ \vdots & \vdots &  & \ddots & & \vdots\\ 0 & 0 & \cdots &0& 1 & 0 \\ 0 & 0& 0 &\cdots & 0 &  1 \\ a_1 & a_2 & a_3 & \cdots & \cdots & a_n \end{bmatrix} \quad \bm B = \begin{bmatrix}0 \\ 0 \\ \vdots \\ 0 \\ 0 \\ 1\end{bmatrix}.$$ Let $\bm x^d = \begin{bmatrix}x^d_1 & 0 &\cdots & 0 \end{bmatrix}^T$ be the desired state. We propose a continuous, state-feedback control law $u$ so that $\exists \quad T<\infty$ such that $\forall t\geq T$, $\bm x(t) = \bm x^d$: 
\begin{LT FTS}
System \eqref{lin-sys} reaches the desired state $\bm x^d$ in finite with the control law
\begin{align}\label{LTI FTC}
    u = \dot x_n^d -\sum_{i = 1}^{n}a_ix_i -k_n\sign(x_n-x^d_n)|x_n-x^d_n|^\alpha,
\end{align}
where $k_i>0$, $x_i^d$ is given out of \eqref{xid exp} and $\frac{n-1}{n} <\alpha <1$. Furthermore, the controller $u$ remains bounded. 
\end{LT FTS}
\begin{proof}

System \eqref{lin-sys} can be re-written in the following form: 
\begin{align}\label{n-int}
\begin{bmatrix}\dot x_1\\ \vdots \\ \dot x_{n-1} \\ \dot x_n \end{bmatrix}=\begin{bmatrix}x_2 \\ \vdots \\ x_n \\ \sum_{i=1}^{n}a_ix_i + u\end{bmatrix}.
\end{align}
For $x_1 \rightarrow x^d_1$ in finite-time, the desired time-rate of $x_1$, i.e. desired $x_2$ (denoted as $x^d_2$) should be 
\begin{align*}
x_2^d = -k_1\sign(x_1-x^d_1)|x_1-x^d_1|^\alpha + \dot x^d_1,
\end{align*}
where $k_1>0$ and $0<\alpha<1$ (see Lemma \ref{Lemma FT}). As we assume $x^d_1$ to be constant, we have $\dot x^d_1 = 0$. Similarly, in general form, one car write:
\begin{align}\label{xid exp}
x^d_{i+1} = -k_{i}\sign(x_{i}-x^d_i)|x_{i}-x^d_i|^\alpha + \dot x^d_i,
\end{align}
where $1\leq i\leq n-1$ and $k_i>0$. 
Choose candidate Lyapunov function 
\begin{align*}
    V(\bm x) = \sum\limits_{i = 1}^{n}\frac{1}{2}(x_i-x^d_i)^2
\end{align*}
Taking the time derivative of the candidate Lyapunov function along the closed loop system trajectory with controller \eqref{LTI FTC}, we get
\begin{align*}
    \dot V(\bm x) & = \sum\limits_{i = 1}^{n-1}(x_i-x^d_i)(x_{i+1}-\dot x^d_i) \\
    & + (x_n-x^d_n)(u + \sum\limits_{j = 1}^n a_jx_j-\dot x^d_n)\\
    & = \sum\limits_{i = 1}^{n-1}(x_i-x^d_i)(x_{i+1}-x^d_{i+1}+ x_{i+1}^d-\dot x^d_i) \\
    & + (x_n-x^d_n)(u + \sum\limits_{j = 1}^n a_jx_j-\dot x^d_n)\\
    & = \sum\limits_{i = 1}^{n-1}(x_i-x^d_i)(x_{i+1}- x^d_{i+1})\\
    & + \sum\limits_{i = 1}^{n-1}(x_i-x^d_i)(-k_i\sign(x_i-x^d_i)|x_i-x^d_i|^\alpha) \\
    & + (x_n-x^d_n)(-k_n\sign(x_n-x^d_n)|x_n-x^d_n|^\alpha)\\
    = & \sum\limits_{i = 1}^{n-1}(x_i-x^d_i)(x_{i+1}- x^d_{i+1})+ \sum\limits_{i = 1}^{n}-k_i|x_i-x^d_i|^{\alpha+1}\\
    & \leq V(\bm x) -\sum\limits_{i = 1}^{n}k_i|x_i-x^d_i|^{\alpha+1},
\end{align*}
as $\sum\limits_{i = 1}^{n-1}(x_i-x^d_i)(x_{i+1} \leq \frac{1}{2}\|\bm x_e\|^2 = V(\bm x)$. Define $\bar k = \min \limits_i k_i$ so that we get
\begin{align*}
    \dot V(\bm x) & \leq V(\bm x) -\bar k\sum\limits_{i = 1}^{n}|x_i-x^d_i|^{\alpha+1} = V(\bm x) -\bar k \|\bm x_e\|_{1+\alpha}^{1+\alpha},
\end{align*}
where $\bm x_e \triangleq \begin{bmatrix} x_1-x_1^d & x_2-x_2^d & \dots & x_n-x_n^d \end{bmatrix}^T$ and $\|\bm x_e\|_{1+\alpha}^{1+\alpha}$ is $(1+\alpha)-$norm of vector $\bm x_e$ raised to power $(1+\alpha)$.  Using the norm inequality for equivalent norms, we have that $\|\bm x_e\|_2\leq \|\bm x_e\|_{1+\alpha}$ since $1+\alpha<2$. From this,we get
\begin{align*}
    \dot V(\bm x)\leq  & V(\bm x)-\bar k \|\bm x_e\|_2^{1+\alpha} = V(\bm x)-\bar k (\|\bm x_e\|_2^2)^\frac{1+\alpha}{2}\\
    & \implies \dot V(\bm x) \leq V(\bm x) -cV(\bm x)^\beta,
\end{align*}
where $\beta = \frac{1+\alpha}{2} < 1$ and $c = \bar k 2^\beta$. From \cite{shen2008semi}, we get that in domain $\Omega = \{\bm x \; | \; V(\bm x)^{1-\alpha} < c\}$, the equilibrium point $\bm x^d$ is finite-time stable. Now, in order to be able to stabilize the system in finite-time from any given initial condition, we choose the control gains such that $k_i > \frac{V(\bm x(0))^{1-\alpha}}{2^\beta}$ for all $i$ so that $\bar k> \frac{V(\bm x(0))^{1-\alpha}}{2^\beta}$ and $ V(\bm x(0))^{1-\alpha} < c$. Since $\dot V(\bm x)\leq 0$, we get that $V(\bm x)^{1-\alpha}\leq V(\bm x(0))^{1-\alpha}\leq c$ and hence the closed-loop system would be finite time stable.  Furthermore, with $\alpha >\frac{n-1}{n}$, it can be verified that the controller \eqref{LTI FTC} remain bounded: from \eqref{xid exp}, $\dot x_n^d = -k_{n-1}|x_i-x_d|^{2\alpha-1} + \ddot x_{n-1}^d$. Define $(v)^{(k)}$ as the $k-th$ time derivative of $v$, so that we get $$(x_{i+1}^d)^{
n-i} = -k_i|x_i-x_i^d|^{(n-i+1)\alpha -(n-i)} + (x_i^d)^{n-i+1}.$$ 
Each $(x_{i+1}^d)^{
n-i}$ is bounded if $\alpha >\frac{n-i}{n-i+1}$ and $(x_i^d)^{n-i+2}$ is bounded. $(x_2^d)^{
n-1} = -k_{1}|x_1-x_1^d|^{n\alpha -(n-1)}$ is bounded if $\alpha >\frac{n-1}{n}$. Hence, with this choice of $\alpha$, all the derivatives $(x_i^d)^{(n-i+2)}$ and the controller remain bounded.
\end{proof}
This shows that any controllable LTI system can be controlled (and trivially, stabilized to origin) in finite-time from any initial condition. 

\section{Simulations}\label{Simulations}
\subsection{Simulation results for Section \ref{Sec 2 Path}}
We consider a sinusoidal trajectory as the desired trajectory, i.e. $\bm r_g(t) = \begin{bmatrix}t & \cos(t)\end{bmatrix}^T$.  Figure \ref{fig:Err Sec2} shows the errors or deviations of coordinates $x(t)$ and $y(t)$ from the desired coordinates $x_g(t)$ and $y_g(t)$. Figure \ref{fig:Traj Sec 2} shows the actual and desired trajectory for the closed-loop system. It can be seen from the figures that trajectory $\bm r(t) = (x(t),y(t))$ converges to the desired trajectory $\bm r_g(t)$ in finite time. 
\begin{figure}[h]
	\centering
	\includegraphics[width=1\columnwidth,clip]{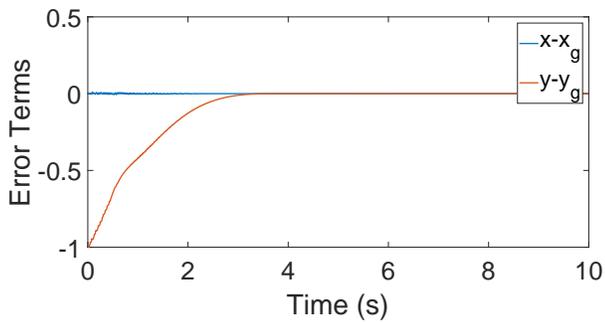}
	\caption{Deviation from Desired Trajectory with time.}
	\label{fig:Err Sec2}
\end{figure}

\begin{figure}[h]
	\centering
	\includegraphics[width=1\columnwidth,clip]{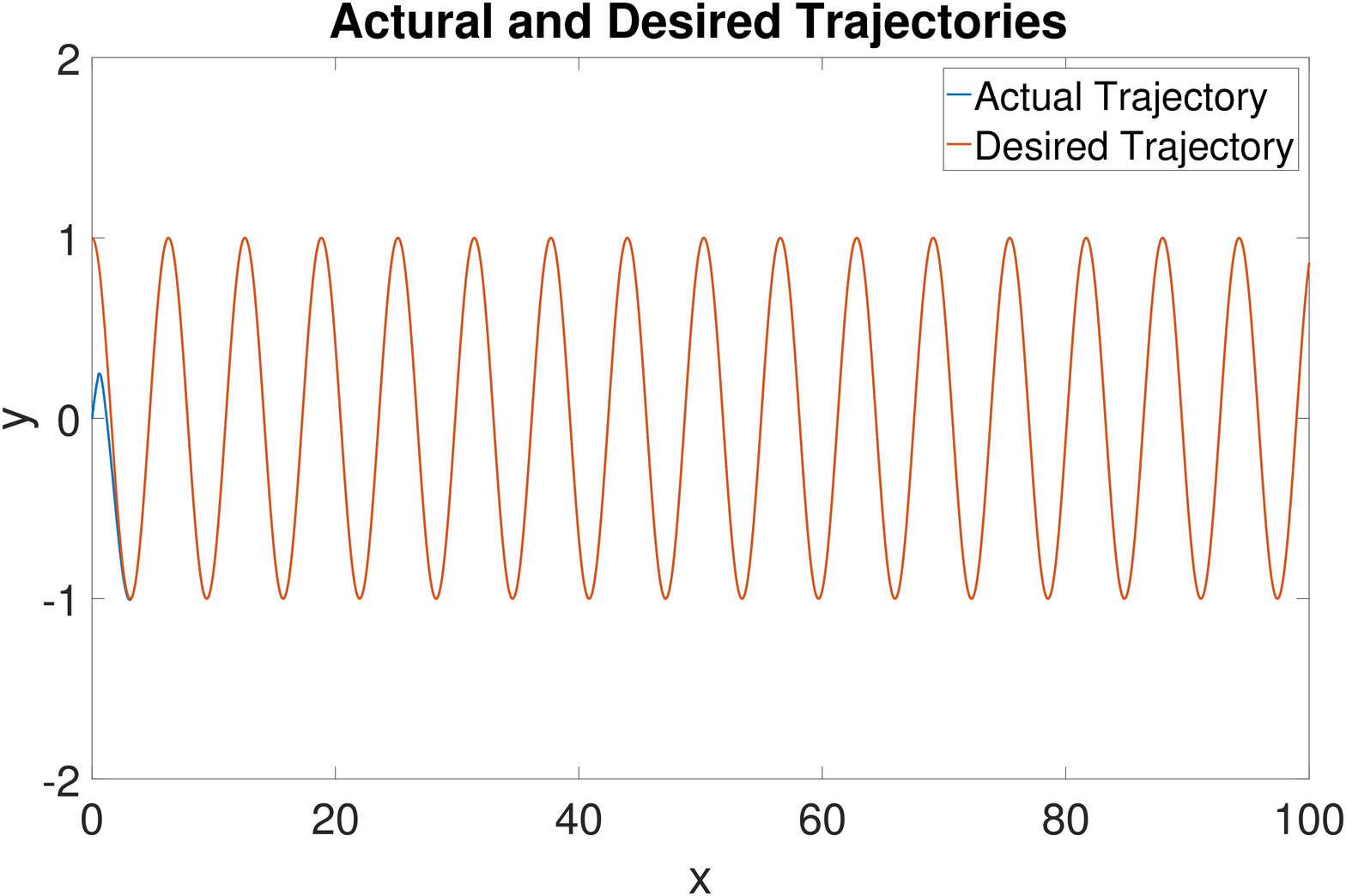}
	\caption{Actual and Desired Trajectories}
	\label{fig:Traj Sec 2}
\end{figure}

\subsection{Simulation results for Section \ref{BF control}}
We consider the desired goal location for the system \eqref{sin int} as $\bm \tau = \begin{bmatrix}10 & 20\end{bmatrix}^T$ and the obstacle at $\bm o = \begin{bmatrix}4 & 6\end{bmatrix}^T$ of radius 1. We use the safe distance $d_c = 2$. Figure \ref{fig:BF DO} shows the path of the vehicle. Figure \ref{fig:BF DG} shows the distance of the vehicle from its goal location, which becomes $0$ in finite time. 
 \begin{figure}[h]
 	\centering
 	\includegraphics[width=0.8\columnwidth,clip]{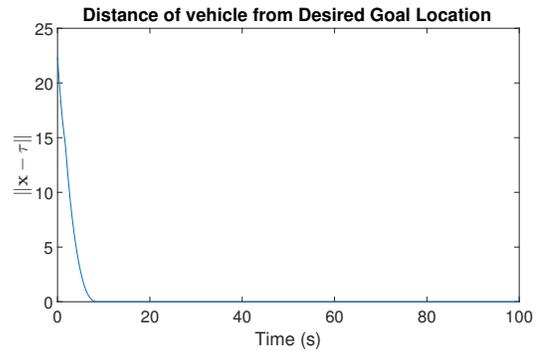}
 	\caption{Distance of the vehicle from desired goal location.}
 	\label{fig:BF DG}
 \end{figure}

\begin{figure}[h]
	\centering
	\includegraphics[width=1\columnwidth,clip]{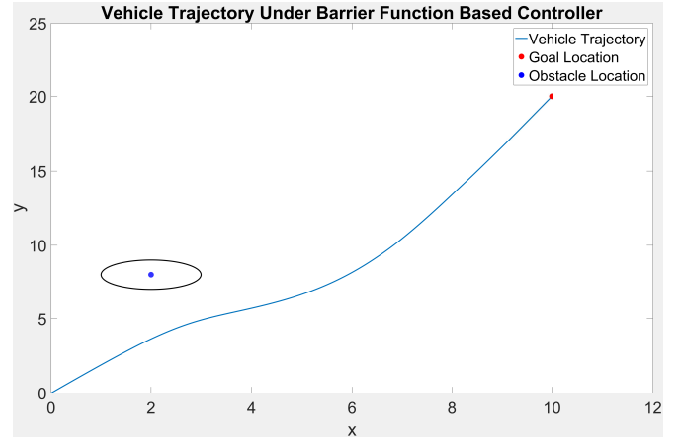}
	\caption{Vehicle Path in presence of Obstacle.}
	\label{fig:BF DO}
\end{figure}

\subsection{Simulation results for Section \ref{LCS}}
We consider 4 states for the system \eqref{lin-sys}. The desired location is chosen as $\bm x_d = [5, 0, 0, 0]^T$. Figure \ref{fig:LCS Traj} shows the trajectory of the system. Again, it can be seen from the figure that the system converges to its desired state in finite time. 

\begin{figure}[h]
	\centering
	\includegraphics[width=1\columnwidth,clip]{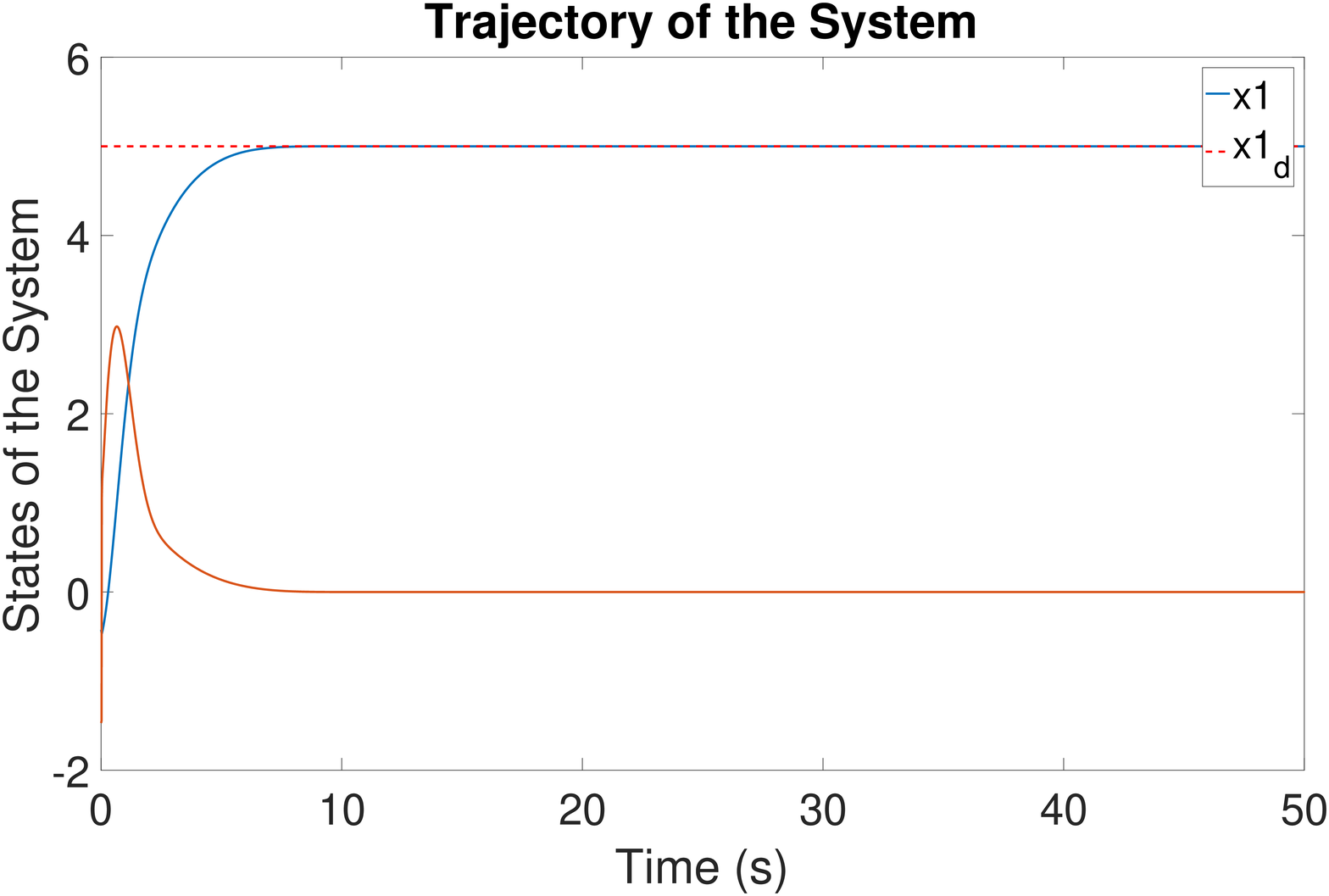}
	\caption{Time Evolution of the system states with time.}
	\label{fig:LCS Traj}
\end{figure}


\section{Conclusions and Future Work}\label{Conclusions}
We presented new geometric conditions for scalar systems in terms the system dynamics evaluated to establish finite-time stability. We demonstrated the utility of the condition through 2 examples where a vector-field based controller is designed for finite-time convergence. We also presented a novel method of designing finite-time Barrier function based control law for obstacle avoidance. Finally, we presented a novel continuous finite-time feedback controller for a general class of linear controllable systems. Our current research focuses on Hybrid and Switched systems. Therefore, in future, we would like to devise condition equivalent to Branicky's condition for Switching systems to be finite-time stable under arbitrary switching. Also, we would also like to expand our collection of finite-time controllers for a general class of non-linear systems. 

\bibliographystyle{IEEEtran}
\bibliography{myreferences}

\end{document}